\documentclass[conference]{IEEEtran}
\IEEEoverridecommandlockouts
\usepackage{cite}
\usepackage{amsmath,amssymb,amsfonts}
\usepackage{algorithmic}
\usepackage{graphicx}
\usepackage{textcomp}
\usepackage{xcolor}
\usepackage{caption}
\usepackage{subcaption}
\usepackage{pdfpages}
\usepackage{amsthm}
\usepackage{amsmath}
\usepackage{ulem}
\usepackage{comment}
\def\BibTeX{{\rm B\kern-.05em{\sc i\kern-.025em b}\kern-.08em
    T\kern-.1667em\lower.7ex\hbox{E}\kern-.125emX}}
\usepackage{booktabs}

\newtheorem{conjecture}{Conjecture}
\newtheorem{proposition}{Proposition}
\def\mysize{0.5}

\begin{document}

\title{Comparison of Proximal First-Order Primal and Primal-Dual algorithms via Performance Estimation\\
\thanks{N. Bousselmi is supported by the French Community of Belgium through a FRIA fellowship (F.R.S-FNRS). N. Pustelnik is supported by the Fondation Simone et Cino Del Duca - Institut de France.}
}

\author{\IEEEauthorblockN{Nizar Bousselmi,}
\IEEEauthorblockA{\textit{ICTEAM, UCLouvain}\\
Louvain-la-Neuve, Belgium \\
nizar.bousselmi@uclouvain.be}
\and
\IEEEauthorblockN{Nelly Pustelnik,}
\IEEEauthorblockA{\textit{CNRS, ENS de Lyon}\\
Lyon, France \\
nelly.pustelnik@ens-lyon.fr}
\and
\IEEEauthorblockN{Julien M. Hendrickx,}
\IEEEauthorblockA{\textit{ICTEAM, UCLouvain} \\
Louvain-la-Neuve, Belgium \\
julien.hendrickx@uclouvain.be}
\and
\IEEEauthorblockN{François Glineur,}
\IEEEauthorblockA{\textit{CORE/ICTEAM, UCLouvain}\\
Louvain-la-Neuve, Belgium \\
{francois.glineur@uclouvain.be}}
\and

}

\maketitle

\begin{abstract}
Selecting the fastest algorithm for a specific signal/image processing task is a challenging question. We propose an approach based on the Performance Estimation Problem framework that numerically and automatically computes the worst-case performance of a given optimization method on a class of functions. We first propose a computer-assisted analysis and comparison of several first-order primal optimization methods, namely, the gradient method, the forward-backward, Peaceman-Rachford, and  Douglas-Rachford splittings. We tighten the existing convergence results of these algorithms and extend them to new classes of functions. Our analysis is then extended and evaluated in the context of the primal-dual Chambolle-Pock and Condat-V{\~u} methods.
\end{abstract}


\section{Introduction}
Many applications in signal processing and machine learning \cite{Combettes2011, bach2012optimization, Chambolle_A_2016_an,Pustelnik_N_20016_j-w-enc-eee_wav_bid} require to solve composite problems of the form 
\begin{equation}\label{eq:problem1}
    \min_{x\in \mathcal{H}} f(x) + g(x)
\end{equation}
or
\begin{equation}\label{eq:problem2}
    \min_{x\in \mathcal{H}} f(x) + h(Mx)
\end{equation}
for a finite-dimensional Hilbert space $\mathcal{H}$ and some properties of smoothness and convexity on functions $f$, $g$, or $h$ and a linear operator $M$. A large number of gradient and proximal-based algorithms have been developed to solve such minimization problems \cite{nemirovski2004prox,Combettes2011,condat2013primal, Chambolle_A_2016_an, beck2017first,condat2023proximal}. However, selecting the fastest method according to a target class of functions remains a tedious task, especially when no tight theoretical convergence rate is available.

In this work, we propose to apply the Performance Estimation Problem (PEP) framework \cite{drori2014performance,taylor2017smooth} to \eqref{eq:problem1} and \eqref{eq:problem2}. This analysis allows us to obtain the exact worst-case performance for these optimization schemes, and therefore to identify the fastest strategy for a given class of functions. \medskip

\noindent \textbf{State-of-the-art for algorithmic comparisons} --  The optimization literature is well documented in terms of design of efficient algorithms. However, quantifying the efficiency of the methods is not usually done in a unified way, which makes comparisons difficult. We often have access to the theoretical worst-case performance of a method in a given setting, which can differ from one analysis to another in terms of function classes and performance criteria. For instance, forward-backward splitting has been studied both in terms of convergence of iterate \cite{chen1997convergence} and function values \cite{beck2009gradient}, but usually, the analysis of methods
exists only for one criterion.

There are also plenty of numerical analyses and comparisons of methods on specific instances of \eqref{eq:problem1} or \eqref{eq:problem2}  that illustrate the behavior of the methods but without any theoretical generalization. For instance, one of the most standard image restoration minimization problems involves $f$ as a quadratic (least squares) term and $h = \lambda \Vert\cdot \Vert_1$ where $\lambda$ is a hyperparameter allowing a tradeoff between the data-fidelity term $f$ and the penalization $\Vert M \cdot \Vert_1$. Numerical comparisons are generally performed for a specific instance of $f$ and $\lambda$, while the performance on other problems can be affected by the Lipschitz constant of the gradient of $f$ or the value of the hyperparameter $\lambda$. 

The PEP framework has been used to derive and prove tight rates in different configurations \cite{drori2014performance,taylor2017smooth,taylor2017exact,ryu2020operator}. Currently, PEP can analyze many first-order methods on different function classes and has been recently extended \cite{bousselmi2023interpolation} to classes involving a linear operator such as in \eqref{eq:problem2}. However, the tool has not been often used to compare different methods in order to select the best one.

Preliminary theoretical studies have been proposed in  \cite{briceno2023theoretical, briceno2022convergence} providing theoretical rate comparisons between different first-order algorithms, thus allowing to establish which method is the fastest. However, this type of analysis faces two limitations: (i) there are no guarantees that these bounds are tight and (ii) such a theoretical analysis is tedious to extend to other algorithmic schemes such as primal-dual methods.\medskip

\noindent \textbf{Contributions} -- 
First, we complement the theoretical analysis in \cite{briceno2023theoretical} that provided upper bounds on the worst-case guarantees for problem \eqref{eq:problem1} with $f$ smooth and strongly convex and $g$ smooth and convex. For primal algorithms, we show the tightness of these bounds with PEP and improve them when possible. Then, we extend their results with numerically supported conjectures for additional classes of functions, namely, to the doubly strongly convex case and to problem \eqref{eq:problem2}. Finally, we use our tool on primal-dual methods.

\section{PEP analysis for the composite problem $f+g$}
We consider problem \eqref{eq:problem1} when $f$ is $\alpha^{-1}$-smooth $\rho$-strongly convex and $g$ is $\beta^{-1}$-smooth $\mu$-strongly convex with $\alpha,\beta \in ]0,\infty[$ and $\rho,\mu \in [0,\infty[$, and we seek the largest (i.e. worst case) contraction factor $r$ such that, for every $(x,y)\in \mathcal{H} \times \mathcal{H}$, 
\begin{equation}
    ||\Phi_{f,g} x - \Phi_{f,g} y || \leq r || x - y||
\end{equation}
where $\Phi_{f,g}$ denotes the output of an algorithmic iteration  to solve \eqref{eq:problem1}. The corresponding conceptual PEP consists in
\begin{equation}\label{eq:PEP}
\begin{aligned} 
\max_{x,y,f,g} \quad & ||\Phi_{f,g} x - \Phi_{f,g} y ||^2\\
\textrm{s.t.} \quad & f \in \mathcal{F}_{\rho,\alpha^{-1}}, g \in \mathcal{F}_{\mu,\beta^{-1}} \\
  & ||x - y||^2 \leq 1,
\end{aligned}
\end{equation}
where $\mathcal{F}_{\mu,L}$ denotes the class of $L$-smooth $\mu$-strongly convex functions. The optimal value of \eqref{eq:PEP} is then equal to the square worst-case contraction factor $r^2$. 

The seminal work in \cite{drori2014performance} raised the idea of computing tight bounds on the performance of a method by formulating a maximization problem over a class of functions and initial conditions and then solving it numerically. Later, it has been shown in \cite{taylor2017smooth} that the conceptual PEP can often be reformulated as a tractable convex semidefinite program for a large number of methods and function classes. One key idea is to replace the optimization over functions by an optimization over function and gradient values under the constraint that these values are consistent with an actual function in the desired class. Then, one must impose suitable conditions on these discrete variables to guarantee the existence of functions $f$ and $g$ consistent with the discrete points, which are called interpolation conditions. Finally, it is possible to lift the problem as a semidefinite program that can be solved efficiently.
In \cite{taylor2017smooth},  the authors have shown that any fixed-step first-order method (including methods with a proximal operator) on the class of smooth strongly convex functions can also be analyzed by PEP. Analyzing a new class of functions with PEP mainly requires deriving an explicit formulation of the corresponding interpolation conditions. Note that PEP allows to consider any performance criterion and initial condition given by a linear combination of the function values and the scalar products between the gradients and the iterates, e.g. $F(x_K) - F^*$, $||x_K - x^*||^2$, $||\nabla F(x_K)||^2$ where $F:=f+g$ or $F:=f+h\circ M$ and $K$ is the iteration index.

\section{First-order proximal methods to solve \eqref{eq:problem1}}

\subsection{PEP to validate the tightness of theoretical analysis}
First, we consider the cases of convex $g$ (i.e. $\mu=0$). We compare the following five optimization methods and recall for each the upper bound on the worst-case contraction factor $r$ proved in \cite{briceno2023theoretical}:
\begin{itemize}
    \item \textbf{Gradient method (GM):} Let $x \in \mathcal{H}$, $\tau\in ]0,\frac{2}{\alpha^{-1} + \beta^{-1}}[$,
    \begin{align}
        \Phi_{f,g} x & = x - \tau \left(\nabla f(x) + \nabla g(x)\right) \tag{GM}, \\
        r_\text{GM}(\tau) & \leq \max\{ |1-\tau \rho |, |1-\tau (\beta^{-1}+\alpha^{-1} ) | \}. \label{eq:r_GM}
    \end{align}

    \item \textbf{Forward-backward splitting (FBS1):} Let $x\in \mathcal{H}$, $\tau \in ]0,2\alpha[$,
    \begin{align}
        \Phi_{f,g} x & = \mathrm{prox}_{\tau g}\left( x - \tau \nabla  f(x) \right), \tag{FBS1} \\
        r_\text{FBS1}(\tau) & \leq \max\{ |1-\tau \rho|, |1--\tau \alpha^{-1}|\}. \label{eq:r_FBS1}
    \end{align}

    \item \textbf{Forward-backward splitting (FBS2):} Let $x\in \mathcal{H}$, $\tau \in ]0,2\beta]$,
    \begin{align}
        \Phi_{f,g} x & = \mathrm{prox}_{\tau f}\left( x - \tau \nabla  g(x) \right), \tag{FBS2} \\
        r_\text{FBS2}(\tau) & \leq \frac{1}{1+\tau \rho}. \label{eq:r_FBS2}
    \end{align}

    \item \textbf{Peaceman-Rachford splitting (PRS):} Let $x\in \mathcal{H}$, $\tau > 0$,
    \begin{equation}\tag{PRS}
        \begin{cases}
            y & = \mathrm{prox}_{\tau f}(x) \\
            \Phi_{f,g} x & = 2 \mathrm{prox}_{\tau g}(2y-x) - 2y + x 
        \end{cases}
    \end{equation}
    \begin{equation}\label{eq:r_PRS}
        r_\text{PRS}(\tau) =\max\limits_{a\in\{\rho,\alpha^{-1}\}} \frac{|1-\tau a|}{1+\tau a}.
    \end{equation}
    \item \textbf{Douglas-Rachford splitting (DRS):} Let $x\in \mathcal{H}$, $\tau > 0$,
    \begin{equation}\tag{DRS}
        \begin{cases}
            y & = \mathrm{prox}_{\tau f}(x) \\
            \Phi_{f,g}x & = \mathrm{prox}_{\tau g}(2y-x) - y + x
        \end{cases}
    \end{equation}
  \begin{align}\label{eq:r_DRS}
  \begin{split}
       r_\text{DRS}(\tau) \leq \min \left \{ \frac{1+r_\text{PRS}}{2},  \frac{1+\tau^2 \rho \beta^{-1}}{(1+\tau \rho)(1+\tau \beta^{-1})}  \right\}. 
   \end{split}
    \end{align}
\end{itemize}
Figure \ref{fig:fig1} compares the above theoretical rates and those obtained with the PEP analysis described in \eqref{eq:PEP}. The results are displayed for three selected configurations of parameters $(\alpha, \beta,\rho)$, but are representative of a large number of additional experiments, which are not displayed here.   

The numerical results obtained by PEP exactly match the above upper bounds for the first four methods (i.e. GM, FBS1, FBS2, PRS). Moreover, inspection of the PEP results allowed to identify simple, explicit functions that attain these bounds, allowing us to formally establish their tightness.
\begin{proposition}
    Each upper bound \eqref{eq:r_GM}, \eqref{eq:r_FBS1}, \eqref{eq:r_FBS2}, and \eqref{eq:r_PRS} is attained by a quadratic function, and hence describes the exact (unimprovable) worst-case contraction factors of (GM), (FBS1), (FBS2), and (PRS).
    \begin{proof}
        The quadratic functions are $f(x) = a \frac{x^2}{2}$ and $g(x) = b \frac{x^2}{2}$ with $a\in \{ \rho, \alpha^{-1} \}$ and $b\in \{0, \beta^{-1} \}$.
    \end{proof}
\end{proposition}
For the fifth method, DRS, Figure \ref{fig:fig1} shows that the theoretical bound \eqref{eq:r_DRS} is conservative and can thus be improved. A direct practical consequence of this conservatism is that it may lead us to make sub-optimal choices of methods. For instance, in Figure~\ref{fig:fig1b}, the upper bounds obtained by \cite{briceno2023theoretical} suggested to use PRS with the step-size $\tau=1$ (black circle). However, DRS behaves much better than predicted by the existing bound in this case. Indeed, given the exact results provided by PEP among all methods, DRS achieves a better convergence rate when considering a step-size $\tau=3.3$ (black square in Figure~\ref{fig:fig1b}). The configuration displayed in Figure~\ref{fig:fig1b} clearly illustrates the drawback of relying on a non-tight analysis of the methods.

\begin{figure}
     \centering
     \begin{subfigure}[b]{\mysize\textwidth}
         \centering
         \includegraphics[width=1\textwidth]{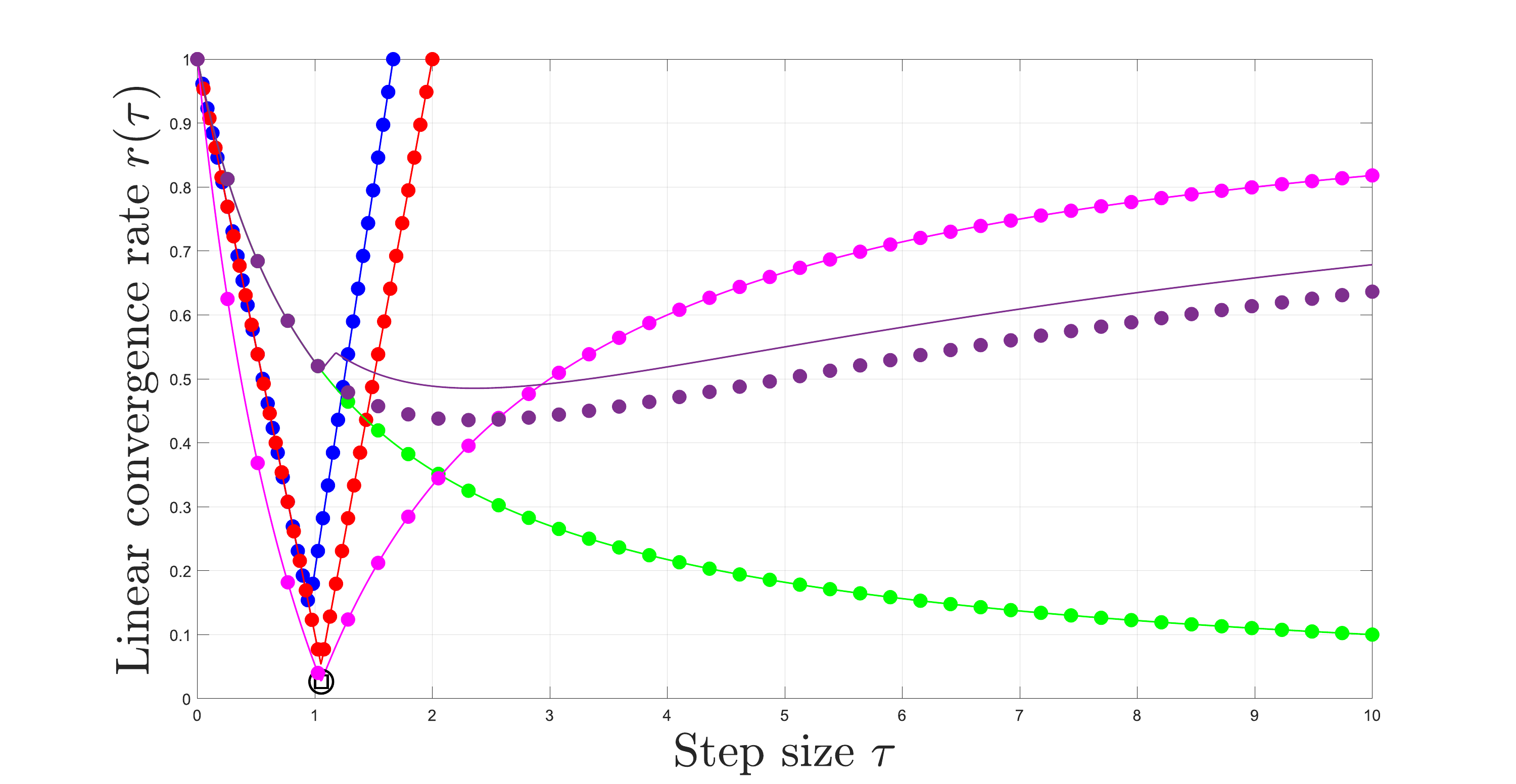}
        \caption{$\alpha=1$, $\beta = 5$, $\rho = 0.9$, $\mu = 0$.}
        \label{fig:fig1a}
     \end{subfigure}
     \hfill
     \begin{subfigure}[b]{\mysize\textwidth}
         \centering
         \includegraphics[width=1\textwidth]{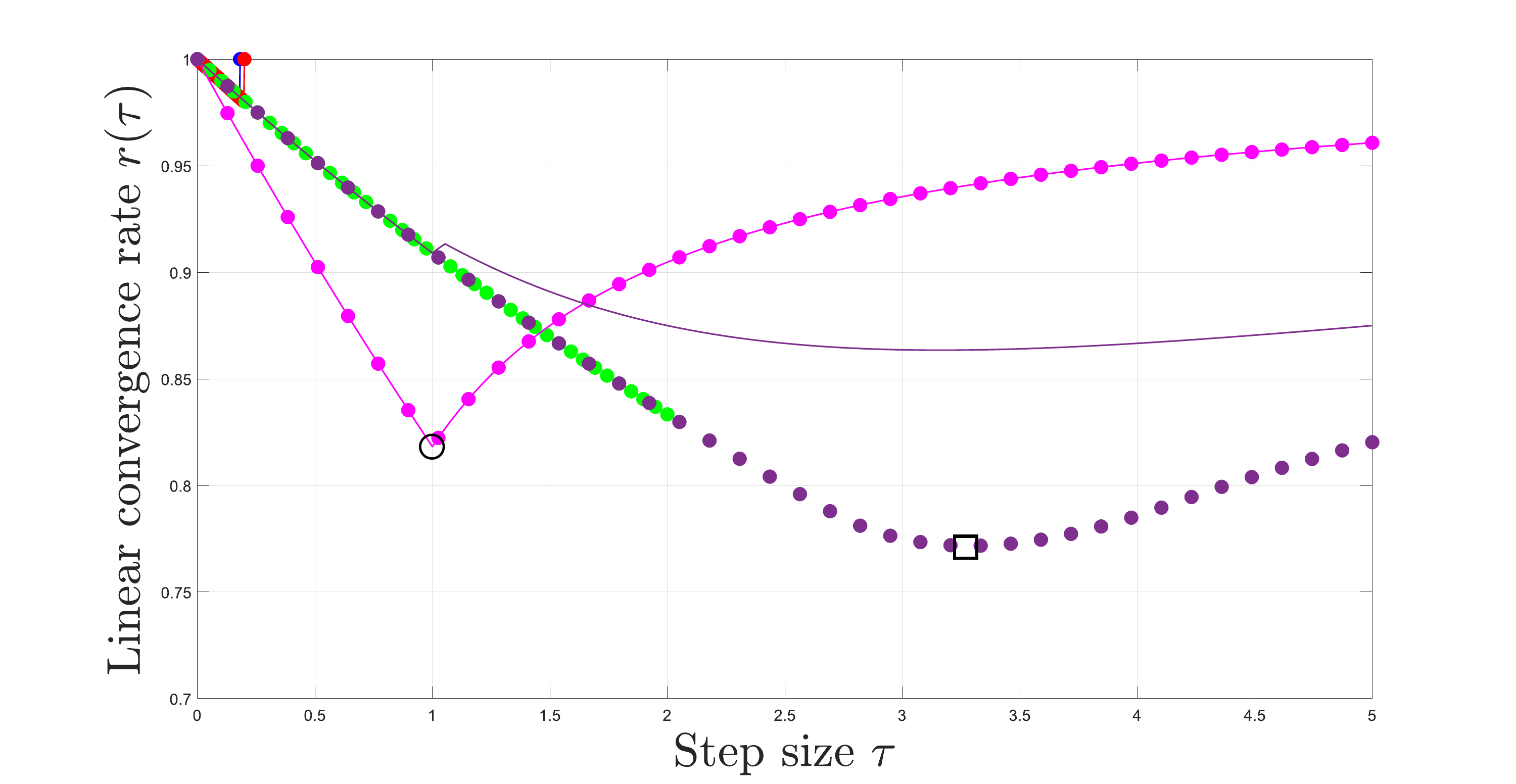}
        \caption{$\alpha=0.1$, $\beta = 1$, $\rho = 0.1$, $\mu = 0$.}
        \label{fig:fig1b}
     \end{subfigure}
     \hfill
      \begin{subfigure}[b]{\mysize\textwidth}
         \centering
         \includegraphics[width=1\textwidth]{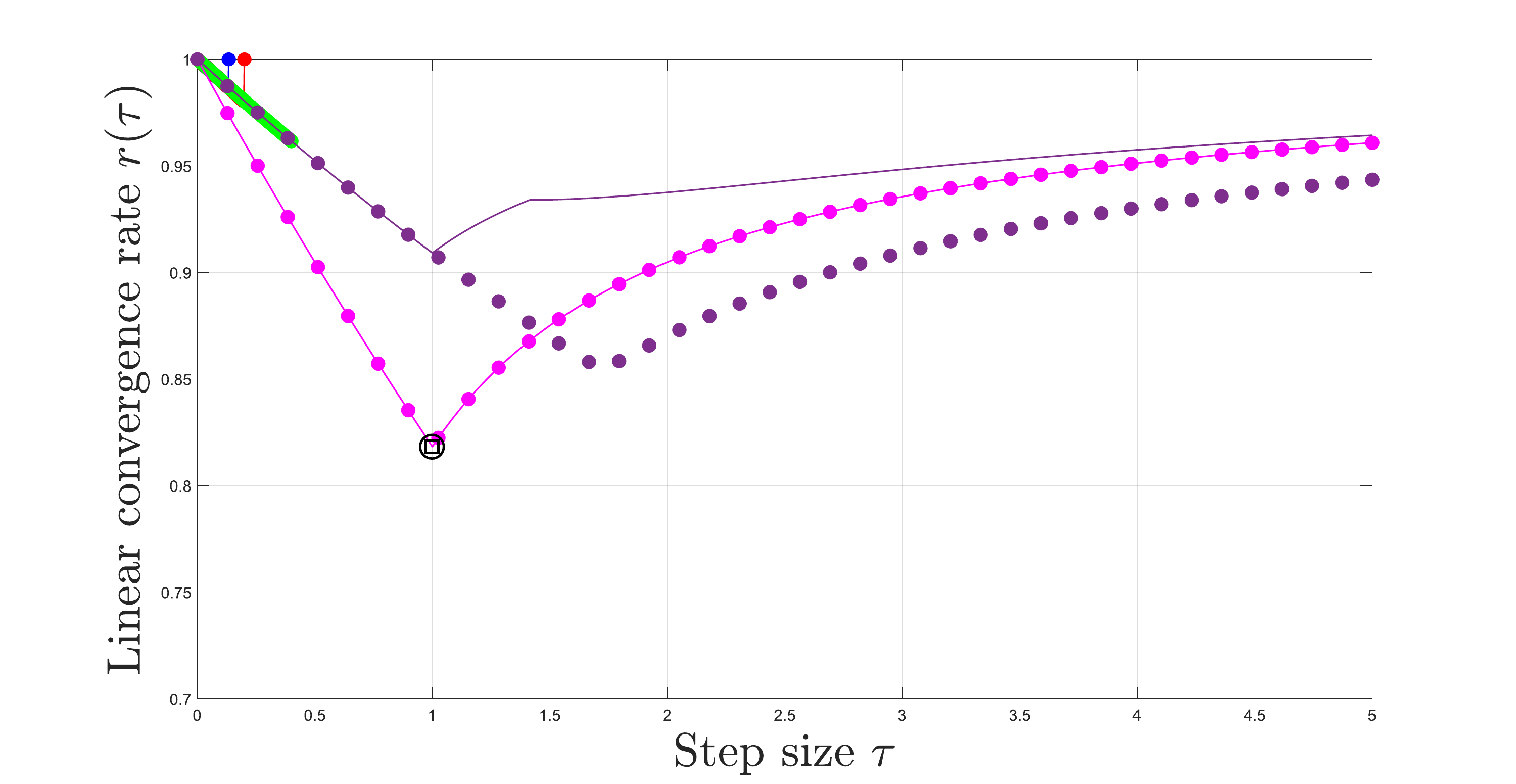}
        \caption{$\alpha=0.1$, $\beta = 0.2$, $\rho = 0.1$, $\mu = 0$.}
        \label{fig:fig1c}
     \end{subfigure}
     \setlength{\abovecaptionskip}{-5pt}
     \setlength{\belowcaptionskip}{-15pt}
    \caption{Comparison of the contraction factors from \cite{briceno2023theoretical} (solid lines) and PEP (dots) of GM (blue), FBS1 (red), FBS2 (green), PRS (magenta), DRS (purple), and the optimal rate predicted by \cite{briceno2023theoretical} (black circle) and PEP (black square).
    }
    \label{fig:fig1}
\end{figure}

\subsection{PEP to evaluate the impact of changing the regularization parameter in the objective function}
Knowing the exact worst-case contraction factor of the methods also allows us to understand what is happening when changing the regularizing hyperparameter in $f + \lambda h$, since it can be seen as an instance of our problem with $g = \lambda h$. The value of $\lambda$ directly impacts the smoothness parameter of $g$. For instance, between Figure \ref{fig:fig1b} and \ref{fig:fig1c}, parameter $\beta$ is divided by 5, which happens when hyperparameter $\lambda$ is multiplied by 5. We observe that, as expected, the behavior of FBS2, PRS, and DRS are impacted, and, considering the PEP results, the optimal method is DRS when $\beta=1$ and PRS when $\beta=5$.

\subsection{Focus on Douglas-Rachford}
Following our observation of non-tightness for DRS, we now investigate this method in more detail.
As illustrated in Figure \ref{fig:fig2}, the theoretical analysis in \cite{briceno2023theoretical} finds one of the regimes (displayed in blue) of performance for shorter step sizes. PEP allows us to identify an additional regime (displayed in red) for larger step sizes. Moreover, by matching the numerical results and the known analytical performance of the method on simple functions (see Section \ref{sect:quad}), we find that the worst-case contraction factor to be given by $\frac{1+\tau^2 \alpha^{-1} \beta^{-1}}{(1+\tau \alpha^{-1})(1+\tau \beta^{-1})}$. However, the precise behavior of the contraction factor in the middle range (cf. Figure \ref{fig:fig2} around $\tau = 7$) remains unidentified.
\begin{proposition}
    The exact worst-case contraction factor $r_\text{DRS}$ of DRS satisfies \vspace{-0.2cm}
    \begin{equation}
        r_\text{DRS}(\tau) \geq \max\Biggl\{\max_{\substack{a\in\{\rho,\alpha^{-1}\} \\ b\in\{0, \beta^{-1} \} }} \frac{1+\tau^2 ab}{(1+\tau a)(1+\tau b)}, T_0(\tau) \Biggr\}
    \end{equation}
    where $T_0(\tau)$ is unidentified.
    \begin{proof}
        The bound is reached by quadratic functions $f(x) = a \frac{x^2}{2}$, $g(x) = b \frac{x^2}{2}$ with $a\in \{ \rho, \alpha^{-1} \}$, $b\in \{0, \beta^{-1} \}$ and  unidentified functions $f(x) = s_0(x)$ and $g(x) = t_0(x)$ such that $||\phi_{s_0,t_0}x - \phi_{s_0,t_0}y || = T_0(\tau) ||x-y||$ for some $x,y$.
    \end{proof}
\end{proposition}

\begin{figure}
     \centering
     \includegraphics[width=\mysize\textwidth]{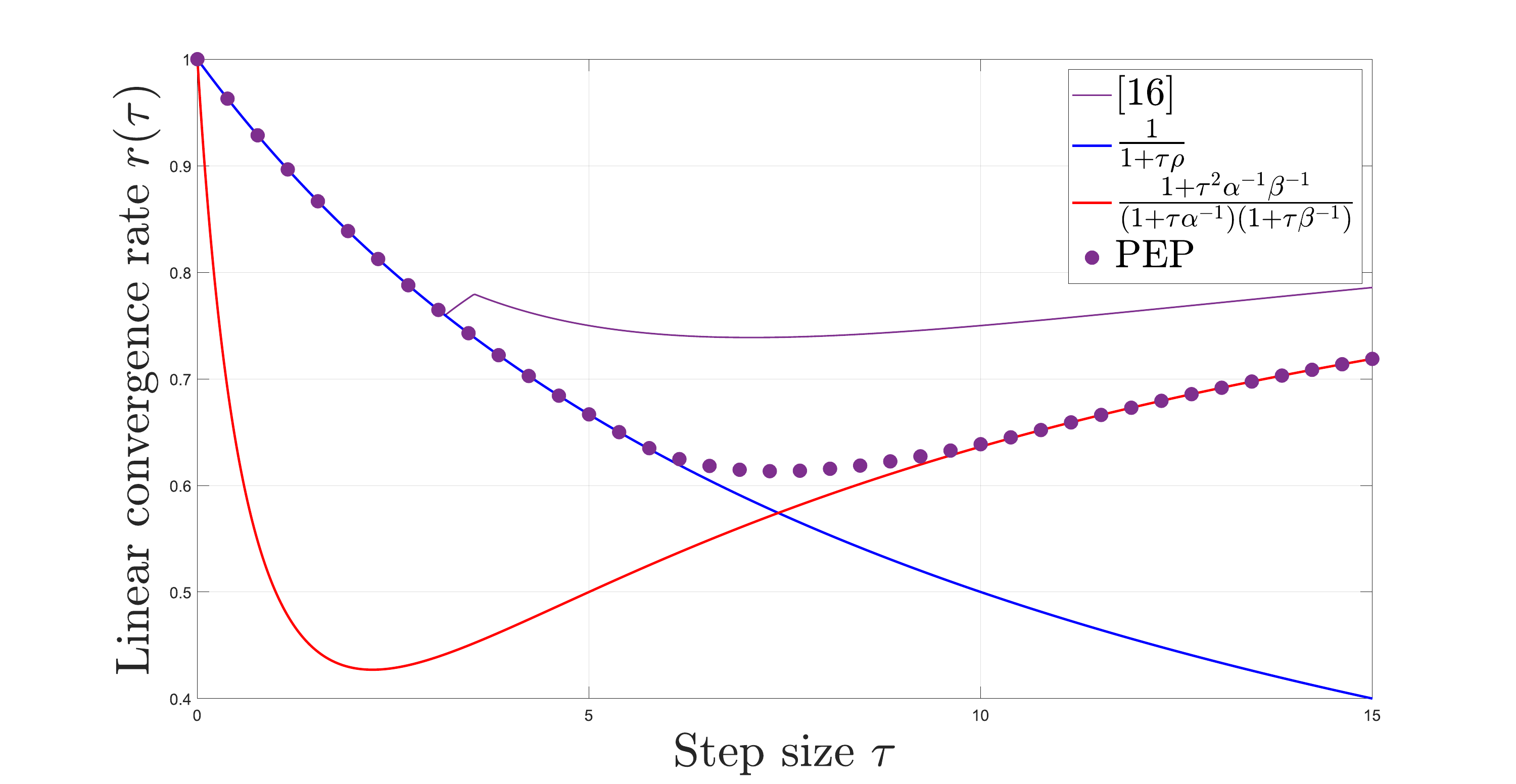}
     \setlength{\belowcaptionskip}{-16pt}
     \setlength{\abovecaptionskip}{-7pt}
    \caption{Worst-case contraction factor of DRS for $\alpha=1$, $\beta = 5$, $\rho = 0.1$, $\mu=0$ from \cite{briceno2023theoretical} (purple solid line) and PEP (purple dots). We identified two regimes (blue and red solid lines).}
    \label{fig:fig2}
 \end{figure}

\subsection{Generalization to strongly convex $f$ and $g$}
Based on PEP analysis, we extend the algorithmic comparison proposed in \cite{briceno2023theoretical} to the case where $g$ is also strongly convex (cf. Figure \ref{fig:fig3}) which was not covered in \cite{briceno2023theoretical}. Similarly to the simply convex case $\mu = 0$, we prove in this setting a lower bound on the worst-case contraction factor for the five methods.  We summarize the identified analytical expressions in the following proposition.
\begin{proposition}\label{prop:proposition_strongly}
    The exact worst-case contraction factor of GM, FBS1, FBS2, PRS, and DRS respectively satisfy
    \begin{align}
        r_\text{GM}(\tau) & \geq \max\{ |1-\tau (\rho + \mu) |, |1-\tau (\beta^{-1}+\alpha^{-1} ) | \}, \\
        r_\text{FBS1}(\tau) & \geq \frac{1}{1+\tau \mu}\max\{ |1-\tau \rho|, |1-\tau \alpha^{-1}|\}, \\
        r_\text{FBS2}(\tau) & \geq\frac{1}{1+\tau \rho}\max\{ |1-\tau \mu|, |1-\tau \beta^{-1}|\},
        \end{align}
        \begin{align}
        r_\text{PRS}(\tau) & \geq\max_{\substack{a\in\{\rho,\alpha^{-1}\} \\ b\in\{\mu, \beta^{-1} \} }} \frac{|1-\tau a|}{1+\tau a} \frac{|1-\tau b|}{1+\tau b}, \\
        r_\text{DRS}(\tau) & \geq \max\left\{\max_{\substack{a\in\{\rho,\alpha^{-1}\} \\ b\in\{\mu, \beta^{-1} \} }} \frac{1+\tau^2 ab}{(1+\tau a)(1+\tau b)}, T_\mu(\tau) \right\}.
    \end{align}
     where $T_\mu$ is unidentified.
     \end{proposition}
     \begin{proof}
         All the bounds are reached by quadratic functions $f(x) = a \frac{x^2}{2}$, $g(x) = b \frac{x^2}{2}$ with $a\in \{ \rho, \alpha^{-1} \}$, $b\in \{\mu, \beta^{-1} \}$ and, for DRS only, unidentified functions $f(x) = s_\mu(x)$, $g(x)=t_\mu(x)$ such that $||\phi_{s_\mu,t_\mu}x - \phi_{s_\mu,t_\mu}y || = T_\mu(\tau) ||x-y||$ for some $x,y$.
     \end{proof}

The numerical results of all our experiments were consistent with the contraction factors used as lower bounds in Proposition \ref{prop:proposition_strongly}. Therefore, we conjecture that they are exact. As for the case $\mu=0$, we were not able to identify the regimes $T_\mu$ of DRS nor the associated worst-case functions $s_\mu$ and $t_\mu$.

\begin{figure}
     \centering
     \includegraphics[width=\mysize\textwidth]{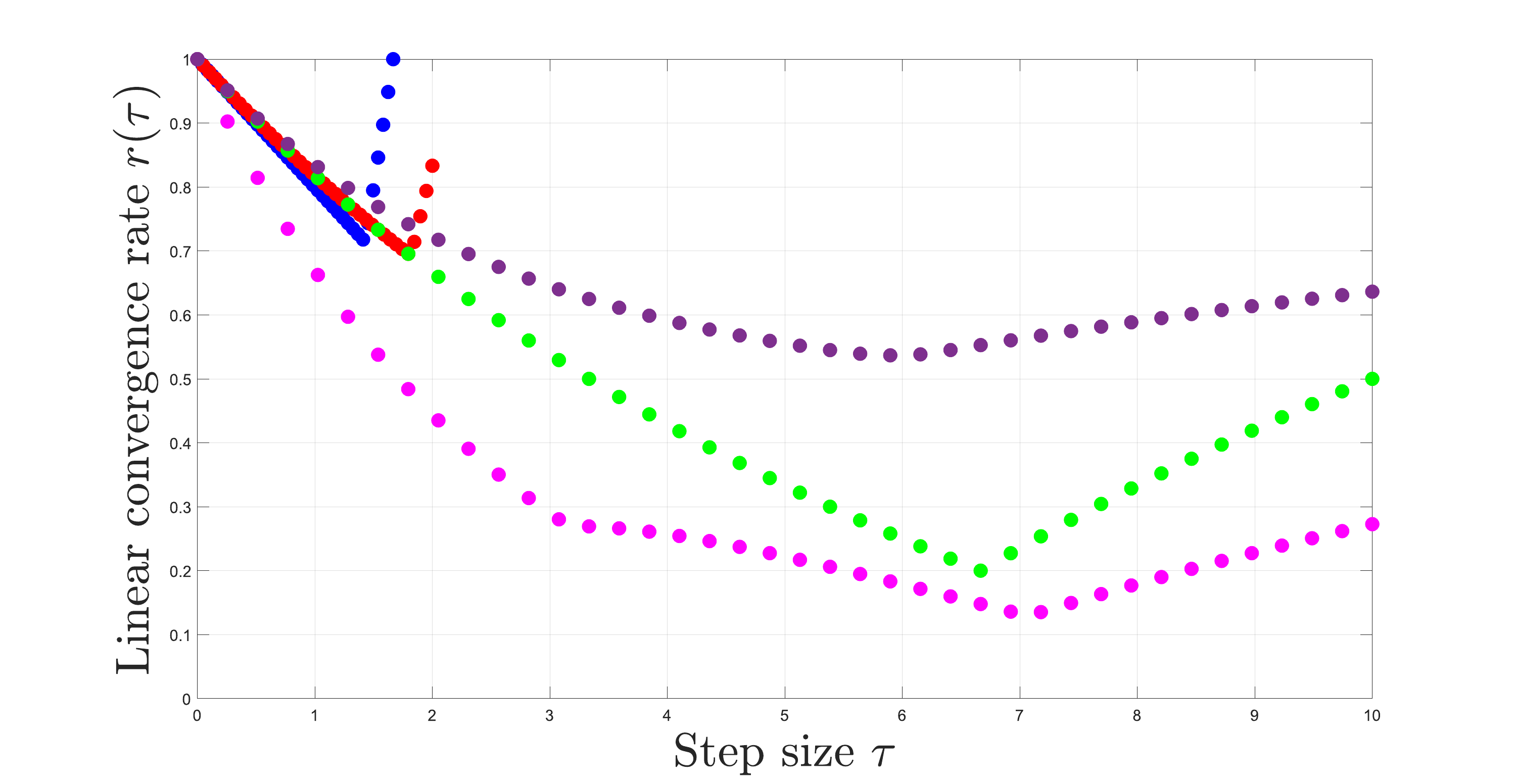}
    \caption{Contraction factors obtained by PEP of GM (blue), FBS1 (red), FBS2 (green), PRS (magenta), DRS (purple) in the doubly strongly convex case $\alpha = 1$, $\beta = 5$, $\rho=\mu=0.1$.
    }
    \label{fig:fig3}
\end{figure}

\section{Linear rate on quadratic functions}\label{sect:quad}
Interestingly, all worst cases identified so far are attained by univariate quadratic functions. The contraction factor of a method on the class of quadratic functions can be explicitly computed by solving a simple maximization problem.

For example, let us consider PRS applied to the problem $f(x)= a\frac{x^2}{2}$ and $ g(x) = b \frac{x^2}{2}$. The proximal step on a quadratic function $f(x) = a \frac{x^2}{2}$ can be explicitly computed as
\begin{align}
     \mathrm{prox}_{\tau f}(x)  = \frac{x}{1+\tau a}.
\end{align}
Therefore, after simplifications, PRS becomes
\begin{equation}
    \Phi_{f,g} x = \frac{1-\tau a}{1+\tau a} \frac{1- \tau b}{1+\tau b}x.
\end{equation}
In this case, problem \eqref{eq:PEP} computing the worst-case contraction factor reduces to 
\begin{equation}
    r^2(\tau) = \max_{\substack{a\in[\rho,\alpha^{-1}] \\ b\in[\mu, \beta^{-1} ] }} \left|\frac{1-\tau a}{1+\tau a} \frac{1- \tau b}{1+\tau b} \right|^2,
\end{equation}
namely, a simple optimization problem involving two variables $a$ and $b$ that can often be solved analytically. The same methodology applies to the other methods and also to general quadratic functions (not necessarily univariate).

This methodology allowed us to find the analytical expressions of all the contraction factors gathered in Proposition \ref{prop:proposition_strongly} (except the unidentified one of DRS). Nonetheless, it is important to note that it was not known a priori that the worst-case contraction factors of the methods over general smooth convex functions are attained by quadratic functions. The analysis and numerical results provided by PEP are critical to reach this conclusion.
\begin{conjecture}
    The worst-case contraction factor of GM, FBS1, FBS2, and PRS on the problem $\min_x f(x) + g(x)$ where $f\in\mathcal{F}_{\rho,\alpha^{-1}}$ and $g\in \mathcal{F}_{\mu,\beta^{-1}}$ are reached by univariate quadratic functions $f$ and $g$.
\end{conjecture}

\section{Case $f+h\circ M$}\label{sect:Mx}
We now move to composite functions of the form \eqref{eq:problem2} involving a linear operator, whose inclusion in the PEP framework was made possible by the very recent work \cite{bousselmi2023interpolation}. Specifically, we assume that $f$ is $\alpha^{-1}$-smooth $\rho$-strongly convex, $h$ is $\gamma^{-1}$-smooth $\delta$-strongly convex, and matrix $M$ such that $||M|| \leq L$ with $\alpha,\rho,L \in ]0,\infty[$ and $\rho,\delta \in [0,\infty[$. This is a particular case of problem \eqref{eq:problem1}, therefore, the previous results still apply. However, the composite structure of $g=h\circ M$ can in principle lead to stronger results. We still consider the five same methods and want to quantify the potential gain of performance when applying a method on \eqref{eq:problem2}. Note that even if computing the proximal operators of the iterations of the method can be complicated for practical problems, it does not prevent us from easily studying them with PEP.

The only setting where an improvement could be expected is for DRS. Indeed, for the four other methods, the worst-case functions on the class of functions of the form $f+g$ are quadratic. Since, all quadratic functions $g$ can be expressed as a composite function of the form $h\circ M$, the worst-case performance cannot be improved. However, since the worst-case functions of DRS are not identified, there is room for possible improvement.  
However, as illustrated in Figure \ref{fig:fig4}, where we depict the results for DRS, there is no improvement. It indicates that there exists a worst-case function in the case $f+g$ for which $g$ is of the form $h\circ M$ (and that is not a quadratic function).

\begin{figure}
     \centering
         \includegraphics[width=\mysize\textwidth]{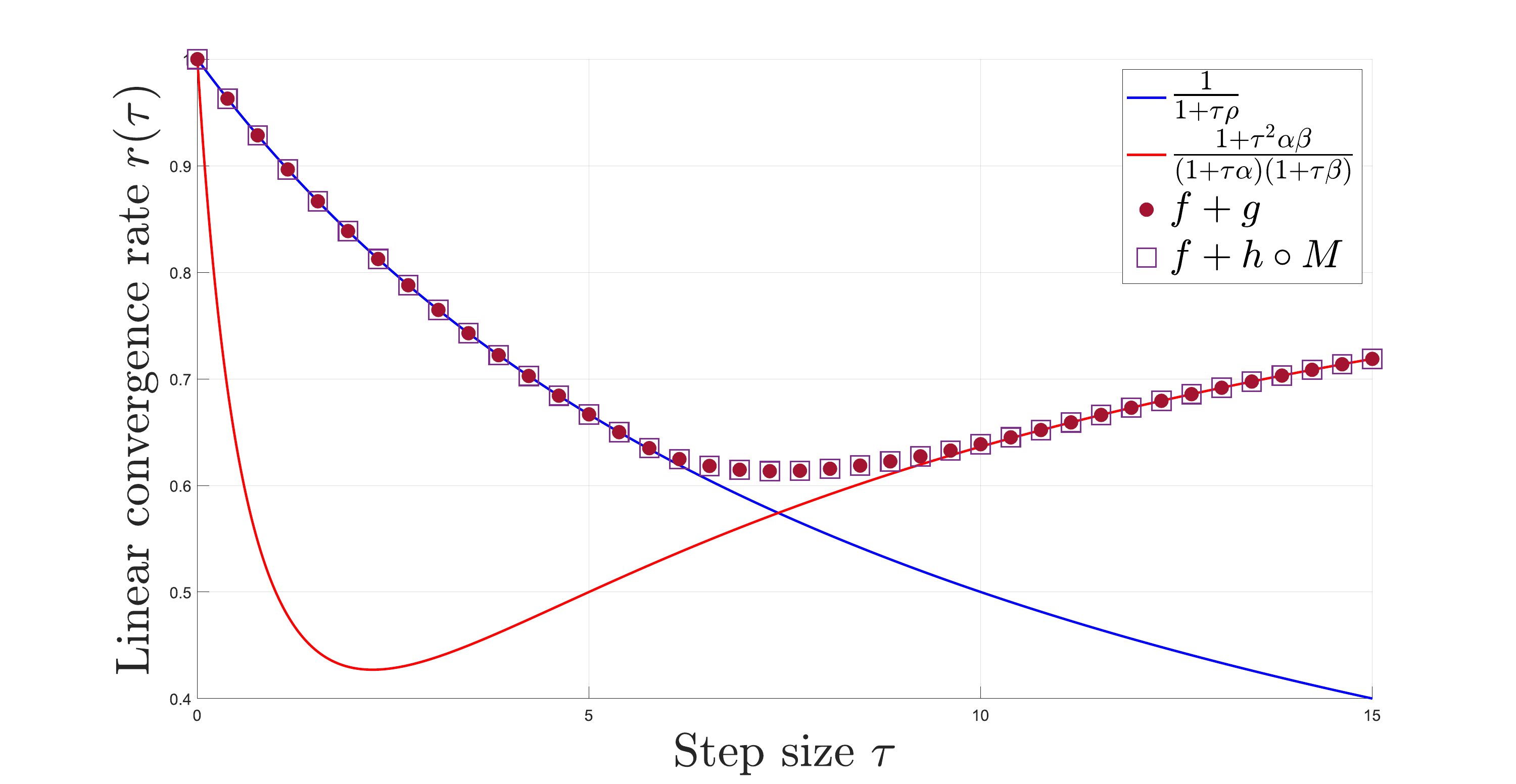}
         \setlength{\abovecaptionskip}{-7pt}
         \setlength{\belowcaptionskip}{-16pt}
        \caption{Contraction factor of DRS provided by PEP on \eqref{eq:problem1} (purple dots) and \eqref{eq:problem2} (purple squares) for $\alpha=1$, $\beta = \gamma = 5$, $\rho = 0.1$, $\mu = 0$, $\delta = 0.1$ and $ L = 1$.}
    \label{fig:fig4}
\end{figure}

\section{Primal-dual methods}
Unlike the results of \cite{briceno2023theoretical} for primal methods, no theoretical comparison for proximal primal-dual algorithms seems to be available in the literature so far. We now show how the recent extension of PEP allows us to numerically compute these rates, and therefore, compare methods. We focus on two primal-dual methods that exploit access to linear operators in their iterations:
\begin{itemize}
\item \textbf{Chambolle-Pock method (CPM)} \cite{chambolle2011first}: Let $x$, $u$, and $\tau,\sigma$ such that $\sigma \tau ||M||^2 \leq 1$:
    \begin{equation}
        \begin{cases}
            x_{+} & = \mathrm{prox}_{\tau f }(x - \tau M^T u) \\
            u_{+} & = \mathrm{prox}_{\sigma h^* }\left(u + \sigma M(2x_{+}-x)\right). 
        \end{cases}
    \end{equation}
\end{itemize}
\begin{itemize}
    \item \textbf{Condat-V{\~u} method (CVM)}\cite{condat2013primal,vu2013splitting}: Let $x$, $u$ and  $\tau,\sigma$ such that $\frac{1}{\tau}-\sigma ||M||^2 \geq \frac{1}{2\alpha}$ (where $f$ is assumed to be $\alpha^{-1}$-smooth):
    \begin{equation}
        \begin{cases}
            x_{+} & = x - \tau \nabla f(x) - \tau M^T u \\
            u_{+} & = \mathrm{prox}_{\sigma h^* }\left(u + \sigma M(2x_{+}-x)\right). 
        \end{cases}
    \end{equation}
\end{itemize}
For such primal-dual methods, we are interested in the contraction factor $r$ such that
\begin{equation}
    ||\phi_{f,g}(x,u) - \phi_{f,g}(x',u')|| \leq r ||(x,u) - (x',u')||.
\end{equation}
Similarly to Section \ref{sect:Mx}, such worst-case analysis can be formulated and obtained with PEP. Figure \ref{fig:fig5} depicts the worst-case contraction factor of CPM and CVM on problem \eqref{eq:problem2} for $\delta = 0$ and $\delta=0.1$. We set $\sigma(\tau) = \frac{1}{\tau ||M||^2}$ for CPM and $\sigma(\tau) = \frac{1}{\tau ||M||^2 } - \frac{1}{2\alpha ||M||^2}$ for CVM.
\begin{figure}
     \centering
         \includegraphics[width=\mysize\textwidth]{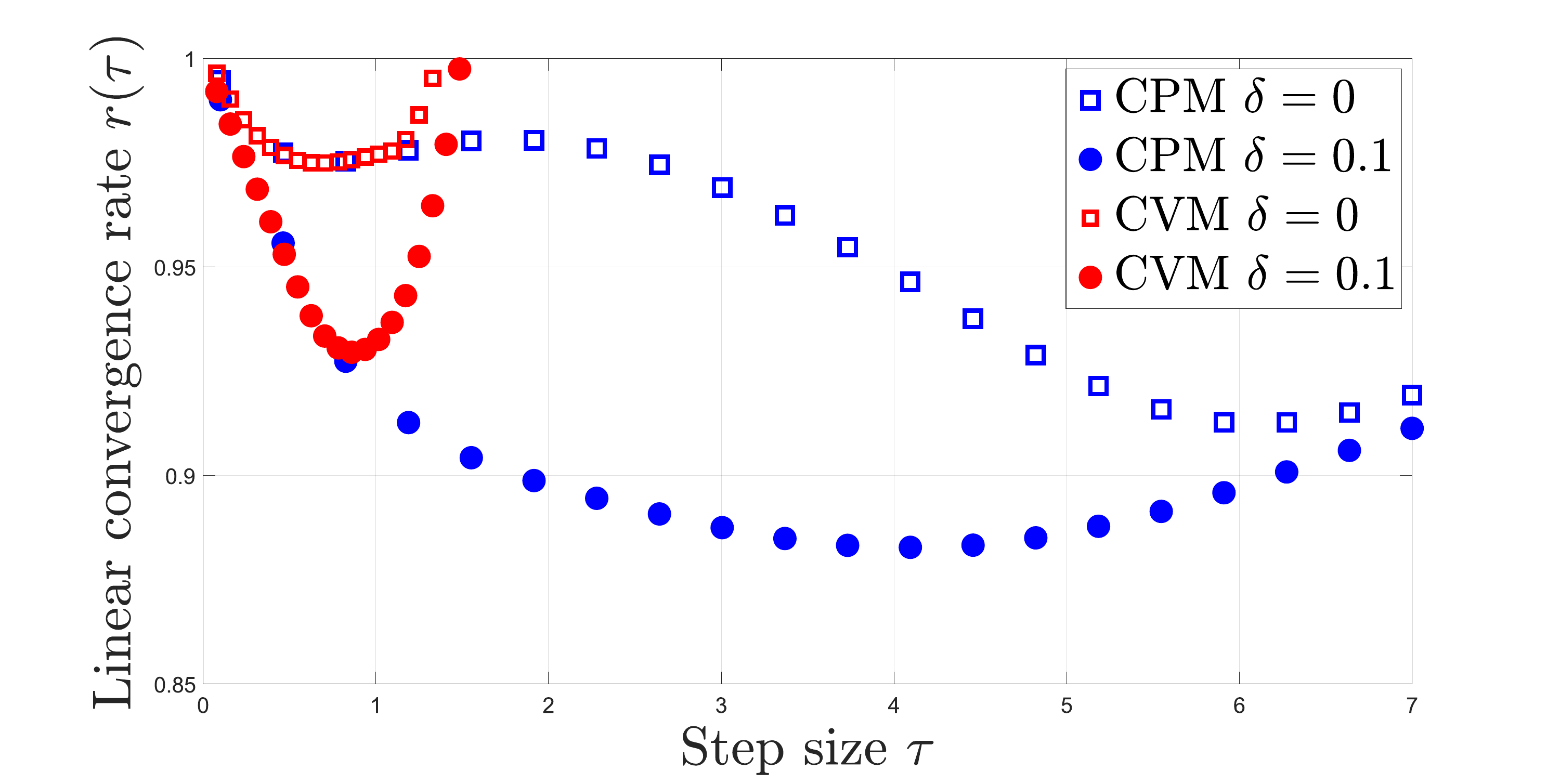}
         \setlength{\abovecaptionskip}{-7pt}
         \setlength{\belowcaptionskip}{-16pt}
        \caption{Contraction factors of CPM (blue) and CVM (red) obtained by PEP for $\alpha = 1$, $\beta=5$, $\rho=0.1$, $\delta=0$ (squares) and $\delta = 0.1$ (dots) and $||M||\leq 1$.}
    \label{fig:fig5}
\end{figure}
We can observe the impact of the strong convexity of $h$. In these two specific settings, CPM performs better.
\section{Conclusion}
We have shown how the PEP framework allows to easily certify the tightness of recent bounds on the contraction factor of the gradient method and the forward-backward, Peaceman-Rachford, and Douglas-Rachford splitting, or to improve them. Performing the same analysis on the doubly strongly convex case is equally straightforward, and again we obtain tight numerical rates. The methodology also highlights interesting and a priori unexpected phenomena, such as the existence of quadratic worst-case functions.

This type of analysis can be extended to any other method, function class, and performance criterion. We demonstrate this possibility on Chambolle-Pock and Condat-V{\~u} methods opening the way for more extended analysis in further works. \vspace{-0.6cm}

\bibliographystyle{IEEEtran} 
\bibliography{references.bib}

\end{document}